\documentclass{amsart}
\usepackage{setspace}
\usepackage{a4}
\usepackage{amsthm}
\usepackage{latexsym}
\usepackage{amsfonts}
\usepackage{graphicx}
\usepackage{textcomp}
\usepackage{cite}
\usepackage{enumerate}
\usepackage{amssymb}
\usepackage{hyperref}
\usepackage{amsmath}
\usepackage{tikz}
\usepackage[mathscr]{euscript}
\usepackage{mathtools}
\newtheorem{theorem}{Theorem}[section]

\newtheorem{corollary}[theorem] {Corollary}
\newtheorem{definition}[theorem]{Definition}
\newtheorem{example}[theorem]{Example}

\newtheorem{question}[theorem]{Question}
\title{This is the title}
\usepackage{fancyhdr}

\begin{document}
	\vspace{0.9cm}	
\hrule\hrule\hrule\hrule\hrule
\vspace{0.3cm}	
\begin{center}
{\bf{FUNCTIONAL GHOBBER-JAMING UNCERTAINTY PRINCIPLE}}\\
\vspace{0.3cm}
\hrule\hrule\hrule\hrule\hrule
\vspace{0.3cm}
\textbf{K. MAHESH KRISHNA}\\
Post Doctoral Fellow \\
Statistics and Mathematics Unit\\
Indian Statistical Institute, Bangalore Centre\\
Karnataka 560 059, India\\
Email: kmaheshak@gmail.com\\

Date: \today
\end{center}

\hrule\hrule
\vspace{0.5cm}
\textbf{Abstract}: Let $(\{f_j\}_{j=1}^n, \{\tau_j\}_{j=1}^n)$ and $(\{g_k\}_{k=1}^n, \{\omega_k\}_{k=1}^n)$ be  two p-orthonormal bases  for a finite dimensional Banach space $\mathcal{X}$. Let  $M,N\subseteq \{1, \dots, n\}$ be such that 
\begin{align*}
	o(M)^\frac{1}{q}o(N)^\frac{1}{p}< \frac{1}{\displaystyle \max_{1\leq j,k\leq n}|g_k(\tau_j) |},
\end{align*}
where $q$ is the conjugate index of $p$. Then for all $x \in \mathcal{X}$,  we show that 
\begin{align}\label{FGJU}
	\|x\|\leq  \left(1+\frac{1}{1-o(M)^\frac{1}{q}o(N)^\frac{1}{p}\displaystyle\max_{1\leq j,k\leq n}|g_k(\tau_j)|}\right)\left[\left(\sum_{j\in M^c}|f_j(x)|^p\right)^\frac{1}{p}+\left(\sum_{k\in N^c}|g_k(x) |^p\right)^\frac{1}{p}\right].
\end{align}
 We call Inequality (\ref{FGJU}) as \textbf{Functional Ghobber-Jaming Uncertainty Principle}. Inequality (\ref{FGJU}) improves the uncertainty principle obtained by Ghobber and Jaming \textit{[Linear Algebra Appl., 2011]}.



\hrule

\tableofcontents
\hrule
\section{Introduction}
Let $d \in \mathbb{N}$ and  $~\widehat{}:\mathcal{L}^2 (\mathbb{R}^d) \to \mathcal{L}^2 (\mathbb{R}^d)$ be the unitary Fourier transform obtained by extending uniquely the bounded linear operator 
\begin{align*}
\widehat{}:\mathcal{L}^1 (\mathbb{R}^d)\cap  \mathcal{L}^2	 (\mathbb{R}^d) \ni f \mapsto \widehat{f} \in  C_0(\mathbb{R}^d); \quad \widehat{f}: \mathbb{R}^d \ni \xi \mapsto \widehat{f}(\xi)\coloneqq \int_{\mathbb{R}^d}	f(x)e^{-2\pi i  \langle x, \xi \rangle}\,dx\ \in \mathbb{C}.
\end{align*}
In 2007, Jaming \cite{JAMING} extended the uncertainty principle obtained by Nazarov for $\mathbb{R}$ in 1993 \cite{NAZAROV} (cf. \cite{HARVINJORICKE}). In the following theorem, Lebesgue measure on $\mathbb{R}^d$ is denoted by $m$. Mean width of a measurable subset  $E$ of $\mathbb{R}^d$ having finite measure  is denoted by $w(E)$.
\begin{theorem}\cite{JAMING, NAZAROV}\label{JN} (\textbf{Nazarov-Jaming Uncertainty Principle}) For each $d \in \mathbb{N}$, there exists a universal constant $C_d$ (depends upon $d$) satisfying the following: If $E, F \subseteq \mathbb{R}^d$ are measurable subsets having finite measure, then for all $f \in \mathcal{L}^2 (\mathbb{R}^d)$, 
	\begin{align}\label{NAZAROVJAMINGIN}
\int_{\mathbb{R}^d}	|f(x)|^2\,dx &\leq C_d e^{C_d \min \{m(E)m(F), m(E)^\frac{1}{d}w(F), m(F)^\frac{1}{d}w(E)\}} 
&\left[\int_{E^c}	|f(x)|^2\,dx+\int_{F^c}	|\widehat{f}(\xi)|^2\,d\xi\right].
	\end{align} 
In particular, if $f$ is supported on $E$ and 	$\widehat{f}$ is supported on $F$, then $f=0$. 
\end{theorem}
Theorem \ref{JN} and the milestone paper \cite{DONOHOSTARK} of Donoho and Stark which derived finite dimensional uncertainty principles, motivated Ghobber and Jaming \cite{GHOBBERJAMING} to ask what is the exact finite dimensional analogue of Theorem   \ref{JN}? Ghobber and Jaming were able to derive the following beautiful theorem. Given a subset  $M\subseteq \{1, \dots, n\}$, the number of elements in $M$ is denoted by $o(M)$. 
\begin{theorem}\cite{GHOBBERJAMING}\label{GJ} (\textbf{Ghobber-Jaming  Uncertainty Principle})
Let 	$\{\tau_j\}_{j=1}^n$ and $\{\omega_j\}_{j=1}^n$ be orthonormal bases for   the Hilbert space $\mathbb{C}^n$.  If $M,N\subseteq \{1, \dots, n\}$ are such that 
\begin{align}
o(M)o(N)< \frac{1}{\displaystyle \max_{1\leq j, k\leq n}|\langle\tau_j, \omega_k\rangle |^2},
\end{align}
then for all $h \in \mathbb{C}^n$, 
\begin{align*}
	\|h\|\leq \left(1+\frac{1}{1-\sqrt{o(M)o(N)}\displaystyle\max_{1\leq j,k\leq n}|\langle\tau_j, \omega_k\rangle |}\right)\left[\left(\sum_{j\in M^c}|\langle h, \tau_j\rangle |^2\right)^\frac{1}{2}+\left(\sum_{k\in N^c}|\langle h, \omega_k\rangle |^2\right)^\frac{1}{2}\right].
\end{align*}
In particular, if $h$ is supported on $M$ in the expansion using basis $\{\tau_j\}_{j=1}^n$ and $h$ is supported on $N$ in the expansion using basis $\{\omega_j\}_{j=1}^n$, then $h=0$. 
\end{theorem}
 It is  reasonable to ask whether there is a Banach space version of Ghobber-Jaming  Uncertainty Principle,  which when restricted to Hilbert space,  reduces to Theorem \ref{GJ}? We are going to answer this question in the paper.

\section{Functional Ghobber-Jaming Uncertainty Principle}
In the paper,   $\mathbb{K}$ denotes $\mathbb{C}$ or $\mathbb{R}$ and $\mathcal{X}$ denotes a  finite dimensional Banach space over $\mathbb{K}$. Identity operator on $\mathcal{X}$ is denoted by $I_\mathcal{X}$. Dual of $\mathcal{X}$ is denoted by $\mathcal{X}^*$. Whenever $1<p<\infty$, $q$ denotes conjugate index of $p$. For $d \in \mathbb{N}$, the standard finite dimensional Banach space $\mathbb{K}^d$ over $\mathbb{K}$ equipped with standard $\|\cdot\|_p$ norm is denoted by $\ell^p([d])$. Canonical basis for $\mathbb{K}^d$ is denoted by $\{\delta_j\}_{j=1}^d$ and $\{\zeta_j\}_{j=1}^d$ be the coordinate functionals associated with $\{\delta_j\}_{j=1}^d$. Motivated from the properties of orthonormal bases for Hilbert spaces, we set the following notion of  p-orthonormal bases which is also motivated from the notion of p-approximate Schauder frames \cite{KRISHNAJOHNSON} and p-unconditional Schauder frames \cite{KRISHNA2}.
\begin{definition}\label{PONB}
	Let $\mathcal{X}$  be a  finite dimensional Banach space over $\mathbb{K}$.   Let $\{\tau_j\}_{j=1}^n$ be a basis for   $\mathcal{X}$ and 	let $\{f_j\}_{j=1}^n$ be the coordinate functionals associated with $\{\tau_j\}_{j=1}^n$. The pair $(\{f_j\}_{j=1}^n, \{\tau_j\}_{j=1}^n)$ is said to be a \textbf{p-orthonormal basis} ($1<p <\infty$) for $\mathcal{X}$ if  the following conditions hold.
	\begin{enumerate}[\upshape(i)]
		\item $\|f_j\|=\|\tau_j\|=1$ for all $1\leq j\leq n$.
		\item For every $(a_j)_{j=1}^n \in \mathbb{K}^n$, 
		\begin{align*}
		\left\|\sum_{j=1}^na_j\tau_j \right\|=\left(\sum_{j=1}^n|a_j|^p\right)^\frac{1}{p}.
		\end{align*}
		\end{enumerate}
\end{definition}
Given a p-orthonormal basis $(\{f_j\}_{j=1}^n, \{\tau_j\}_{j=1}^n)$, we easily see from  Definition \ref{PONB} that 
\begin{align*}
\|x\|=	\left\|\sum_{j=1}^nf_j(x)\tau_j \right\|=\left(\sum_{j=1}^n|f_j(x)|^p\right)^\frac{1}{p}, \quad \forall x \in \mathcal{X}.	
\end{align*}
\begin{example}\label{E}
The pair 	$(\{\zeta_j\}_{j=1}^d, \{\delta_j\}_{j=1}^d)$ is a p-orthonormal basis for $\ell^p([d])$.
\end{example}
Like orthonormal bases for Hilbert spaces, the following theorem characterizes all p-orthonormal bases.
\begin{theorem}\label{ONB}
Let $(\{f_j\}_{j=1}^n, \{\tau_j\}_{j=1}^n)$ be a p-orthonormal basis for $\mathcal{X}$. Then a pair 	$(\{g_j\}_{j=1}^n, \{\omega_j\}_{j=1}^n)$ is a p-orthonormal basis for $\mathcal{X}$ if and only if there is an invertible linear isometry $V:\mathcal{X} \to \mathcal{X}$ such that 
\begin{align*}
	g_j=f_jV^{-1}, ~ \omega_j=V\tau_j, \quad \forall 1\leq j \leq n.
\end{align*}
\end{theorem}
\begin{proof}
	$(\Rightarrow)$ Define $V:\mathcal{X}\ni x \mapsto \sum_{j=1}^{n}f_j(x)\omega_j \in \mathcal{X}$. Since $\{\omega_j\}_{j=1}^n$ is a basis for $\mathcal{X}$, $V$ is invertible with inverse $V^{-1}:\mathcal{X}\ni x \mapsto \sum_{j=1}^{n}g_j(x)\tau_j \in \mathcal{X}$. For $x \in \mathcal{X}$, 
	\begin{align*}
		\|Vx\|=\left\| \sum_{j=1}^{n}f_j(x)\omega_j \right\|=\left(\sum_{j=1}^n|f_j(x)|^p\right)^\frac{1}{p}=\left\| \sum_{j=1}^{n}f_j(x)\tau_j \right\|=\|x\|.
	\end{align*}
	Therefore $V$ is isometry. Note that we clearly have $\omega_j=V\tau_j,  \forall 1\leq j \leq n.$ Now let $1\leq j \leq n$. Then 
	\begin{align*}
		f_j(V^{-1}x)=f_j \left(\sum_{k=1}^{n}g_k(x)\tau_k \right)=\sum_{k=1}^{n}g_k(x)f_j(\tau_k)=g_j(x), \quad \forall x \in \mathcal{X}.
	\end{align*}
$(\Leftarrow)$ Since $V$ is invertible, $\{\omega_j\}_{j=1}^n$ is a basis for $\mathcal{X}$. Now we see that $g_j(\omega_k)=f_j(V^{-1}V\tau_k)=f_j(\tau_k)=\delta_{j,k}$ for all $1\leq j, k \leq n$. Therefore $\{g_j\}_{j=1}^n$ is the coordinate functionals associated with $\{\omega_j\}_{j=1}^n$.  Since $V$ is an isometry, we have $\|\omega_j\|=1$ for all $1\leq j \leq n$. Since $V$ is also  invertible, we have 
\begin{align*}
	\|g_j\|&=\displaystyle\sup_{x\in \mathcal{X}, \|x\|\leq 1}|g_j(x)|=\displaystyle\sup_{x\in \mathcal{X}, \|x\|\leq 1}|f_j(V^{-1}x)|=\displaystyle\sup_{Vy\in \mathcal{X}, \|Vy\|\leq 1}|f_j(y)|\\
	&=\displaystyle\sup_{Vy\in \mathcal{X}, \|y\|\leq 1}|f_j(y)|=\|f_j\|=1, \quad \forall 1\leq j \leq n.
\end{align*}
Finally, for  every $(a_j)_{j=1}^n \in \mathbb{K}^n$, 
\begin{align*}
	\left\|\sum_{j=1}^na_j\omega_j \right\|=	\left\|\sum_{j=1}^na_jV\tau_j \right\|=	\left\|V\left(\sum_{j=1}^na_j\tau_j \right)\right\|=	\left\|\sum_{j=1}^na_j\tau_j \right\|=\left(\sum_{j=1}^n|a_j|^p\right)^\frac{1}{p}.
\end{align*}
\end{proof}
In the next result we  show that Example \ref{E} is prototypical as long as we consider  p-orthonormal bases.
\begin{theorem}
If $\mathcal{X}$	has a p-orthonormal basis $(\{f_j\}_{j=1}^n, \{\tau_j\}_{j=1}^n)$, then $\mathcal{X}$ is isometrically isomorphic to $\ell^p([n])$.
\end{theorem}
\begin{proof}
Define 	$V:\mathcal{X}\ni x \mapsto \sum_{j=1}^{n}f_j(x)\delta_j \in \ell^p([n])$. By doing a similar calculation as in the direct part in the  proof of Theorem \ref{ONB}, we see that $V$ is an invertible isometry.
\end{proof}
Now we derive main result of this paper.
\begin{theorem}\label{FGJ}(\textbf{Functional Ghobber-Jaming  Uncertainty Principle})
Let $(\{f_j\}_{j=1}^n, \{\tau_j\}_{j=1}^n)$	and $(\{g_k\}_{k=1}^n, \{\omega_k\}_{k=1}^n)$ be p-orthonormal bases  for $\mathcal{X}$.  If $M,N\subseteq \{1, \dots, n\}$ are such that 
\begin{align*}
	o(M)^\frac{1}{q}o(N)^\frac{1}{p}< \frac{1}{\displaystyle \max_{1\leq j,k\leq n}|g_k(\tau_j) |},
\end{align*}
then for all $x \in \mathcal{X}$, 	
\begin{align}\label{FGJI}
	\|x\|\leq \left(1+\frac{1}{1-o(M)^\frac{1}{q}o(N)^\frac{1}{p}\displaystyle\max_{1\leq j,k\leq n}|g_k(\tau_j)|}\right)\left[\left(\sum_{j\in M^c}|f_j(x)|^p\right)^\frac{1}{p}+\left(\sum_{k\in N^c}|g_k(x) |^p\right)^\frac{1}{p}\right].
\end{align}
In particular, if $x$ is supported on $M$ in the expansion using basis $\{\tau_j\}_{j=1}^n$ and $x$ is supported on $N$ in the expansion using basis $\{\omega_k\}_{k=1}^n$, then $x=0$. 
\end{theorem}
\begin{proof}
Given $S\subseteq \{1, \dots, n\}$, define 
\begin{align*}
	P_Sx\coloneqq \sum_{j\in S}f_j(x)\tau_j, \quad \forall x \in \mathcal{X}, \quad 
	\|x\|_{S, f}\coloneqq \left(\sum_{j \in S}|f_j(x)|^p\right)^\frac{1}{p}, \quad 
	\|x\|_{S, g}\coloneqq \left(\sum_{j \in S}|g_j(x)|^p\right)^\frac{1}{p}.
\end{align*}
Also define $V:\mathcal{X}\ni x \mapsto \sum_{k=1}^{n}g_k(x)\tau_k \in \mathcal{X}$. Then $V$ is an invertible isometry. Using $V$ we make following important calculations:
\begin{align*}
	\|P_Sx\|=\left\|\sum_{j\in S}f_j(x)\tau_j\right\|=\left(\sum_{j \in S}|f_j(x)|^p\right)^\frac{1}{p}=	\|x\|_{S, f}, \quad \forall x \in \mathcal{X}
\end{align*}
and 
\begin{align*}
	\|P_SVx\|&=\left\|\sum_{j\in S}f_j(Vx)\tau_j\right\|=\left\|\sum_{j\in S}f_j\left(\sum_{k=1}^{n}g_k(x)\tau_k\right)\tau_j\right\|=\left\|\sum_{j\in S}\sum_{k=1}^{n}g_k(x)f_j(\tau_k)\tau_j\right\|\\
	&=\left\|\sum_{j\in S}g_j(x)\tau_j\right\|=\left(\sum_{j \in S}|g_j(x)|^p\right)^\frac{1}{p}=	\|x\|_{S, g}, \quad \forall x \in \mathcal{X}.
\end{align*}
Now let $y \in \mathcal{X}$ be such that $\{j \in \{1, \dots, n\}: f_j(y)\neq 0\} \subseteq M.$
Then $	\|P_NVy\|=\|P_NVP_My\|\leq \|P_NVP_M\|\|y\|$ and 
\begin{align*}
		\|y\|_{N^c, g}=\|P_{N^c}Vy\|=\|Vy-P_NVy\|\geq \|Vy\|-\|P_NVy\|=\|y\|-\|P_NVy\| \geq \|y\|-\|P_NVP_M\|\|y\|.
\end{align*}
Therefore 
\begin{align}\label{INP}
	\|y\|_{N^c, g}\geq (1-\|P_NVP_M\|)\|y\|.
\end{align}
Let $x \in \mathcal{X}$.   Note that $P_Mx$ satisfies $	\{j \in \{1, \dots, n\}: f_j(P_Mx)\neq 0\} \subseteq M.$ Now using  (\ref{INP}) we get 
\begin{align*}
	\|x\|&=\|P_Mx+P_{M^c}x\|\leq \|P_Mx\|+\|P_{M^c}x\|
	\leq \frac{1}{1-\|P_NVP_M\|}	\|P_Mx\|_{N^c, g}+\|P_{M^c}x\|\\
	&=\frac{1}{1-\|P_NVP_M\|}\|P_{N^c}VP_Mx\|+\|P_{M^c}x\|=\frac{1}{1-\|P_NVP_M\|}\|P_{N^c}V(x-P_{M^c}x)\|+\|P_{M^c}x\|\\
	&\leq \frac{1}{1-\|P_NVP_M\|}\|P_{N^c}Vx\|+\frac{1}{1-\|P_NVP_M\|}\|P_{N^c}VP_{M^c}x\|+\|P_{M^c}x\|\\
	&\leq  \frac{1}{1-\|P_NVP_M\|}\|P_{N^c}Vx\|+\frac{1}{1-\|P_NVP_M\|}\|P_{M^c}x\|+\|P_{M^c}x\|\\
	&=\frac{1}{1-\|P_NVP_M\|}\|P_{N^c}Vx\|+\left(1+\frac{1}{1-\|P_NVP_M\|}\right)\|P_{M^c}x\|\\
	&\leq \|P_{N^c}Vx\|+\frac{1}{1-\|P_NVP_M\|}\|P_{N^c}Vx\|+\left(1+\frac{1}{1-\|P_NVP_M\|}\right)\|P_{M^c}x\|\\
	&=\left(1+\frac{1}{1-\|P_NVP_M\|}\right)[\|P_{N^c}Vx\|+\|P_{M^c}x\|]=\left(1+\frac{1}{1-\|P_NVP_M\|}\right)[\|x\|_{N^c,g}+\|P_{M^c}x\|]\\
	&=\left(1+\frac{1}{1-\|P_NVP_M\|}\right)\left[\left(\sum_{j\in M^c}|f_j(x)|^p\right)^\frac{1}{p}+\left(\sum_{k\in N^c}|g_k(x) |^p\right)^\frac{1}{p}\right].
\end{align*}
For $x \in \mathcal{X}$, we  now find 

\begin{align*}
		&\|P_NVP_Mx\|^p=\left\|\sum_{k\in N}f_k(VP_Mx)\tau_k\right\|^p=\left(\sum_{k\in N}|f_k(VP_Mx)|^p\right)^\frac{1}{p}=\sum_{k\in N}\left|(f_kV) \left(\sum_{j\in M}f_j(x)\tau_j\right)\right|^p\\
		&=\sum_{k\in N}\left| \sum_{j\in M}f_j(x)f_k(V\tau_j)\right|^p=\sum_{k\in N}\left| \sum_{j\in M}f_j(x)f_k\left(\sum_{r=1}^{n}g_r(\tau_j)\tau_r\right)\right|^p
		=\sum_{k\in N}\left| \sum_{j\in M}f_j(x)\sum_{r=1}^{n}g_r(\tau_j)f_k(\tau_r)\right|^p\\
		&=\sum_{k\in N}\left| \sum_{j\in M}f_j(x)g_k(\tau_j)\right|^p\leq \sum_{k\in N}\left(\sum_{j\in M}|f_j(x)g_k(\tau_j)|\right)^p
		\leq \left(\displaystyle \max_{1\leq j,k\leq n}|g_k(\tau_j) |\right)^p\sum_{k\in N}\left(\sum_{j\in M}|f_j(x)|\right)^p\\
		&=\left(\displaystyle \max_{1\leq j,k\leq n}|g_k(\tau_j) |\right)^po(N)\left(\sum_{j\in M}|f_j(x)|\right)^p\leq \left(\displaystyle \max_{1\leq j,k\leq n}|g_k(\tau_j) |\right)^po(N)\left(\sum_{j\in M}|f_j(x)|^p\right)^\frac{p}{p}\left(\sum_{j\in M}1^q\right)^\frac{p}{q}\\
		&\leq \left(\displaystyle \max_{1\leq j,k\leq n}|g_k(\tau_j) |\right)^po(N)\left(\sum_{j=1}^n|f_j(x)|^p\right)^\frac{p}{p}\left(\sum_{j\in M}1^q\right)^\frac{p}{q}=\left(\displaystyle \max_{1\leq j,k\leq n}|g_k(\tau_j) |\right)^po(N)\|x\|^p o(M)^\frac{p}{q}.
\end{align*}
Therefore 
\begin{align*}
	\|P_NVP_M\|\leq \displaystyle \max_{1\leq j,k\leq n}|g_k(\tau_j) |  o(N)^\frac{1}{p}o(M)^\frac{1}{q}
\end{align*}
which gives the theorem.
\end{proof}
\begin{corollary}
	Theorem \ref{GJ} follows from Theorem \ref{FGJ}.
\end{corollary}
\begin{proof}
Let $\{\tau_j\}_{j=1}^n$,  $\{\omega_j\}_{j=1}^n$ be two orthonormal bases   for a  finite dimensional Hilbert space $\mathcal{H}$.	Define 
\begin{align*}
	f_j:\mathcal{H} \ni h \mapsto \langle h, \tau_j \rangle \in \mathbb{K}; \quad g_j:\mathcal{H} \ni h \mapsto \langle h, \omega_j \rangle \in \mathbb{K}, \quad \forall 1\leq j\leq n.
\end{align*}
Then $p=q=2$ and $	|f_j(\omega_k)|=|\langle \omega_k, \tau_j \rangle | $  for all $1\leq j, k \leq n.$	
\end{proof}
By interchanging p-orthonormal bases in Theorem  \ref{FGJ} we  get the following theorem. 
\begin{theorem}
(\textbf{Functional Ghobber-Jaming  Uncertainty Principle}) \label{FGJ2}
Let $(\{f_j\}_{j=1}^n, \{\tau_j\}_{j=1}^n)$	and $(\{g_k\}_{k=1}^n, \{\omega_k\}_{k=1}^n)$ be p-orthonormal bases  for $\mathcal{X}$.  If $M,N\subseteq \{1, \dots, n\}$ are such that 
\begin{align*}
	o(M)^\frac{1}{q}o(N)^\frac{1}{p}< \frac{1}{\displaystyle \max_{1\leq j,k\leq n}|f_j(\omega_k) |},
\end{align*}
then for all $x \in \mathcal{X}$, 	
\begin{align*}
	\|x\|\leq \left(1+\frac{1}{1-o(M)^\frac{1}{q}o(N)^\frac{1}{p}\displaystyle\max_{1\leq j,k\leq n}|f_j(\omega_k)|}\right)\left[\left(\sum_{k\in M^c}|g_k(x)|^p\right)^\frac{1}{p}+\left(\sum_{j\in N^c}|f_j(x) |^p\right)^\frac{1}{p}\right].
\end{align*}
In particular, if $x$ is supported on $M$ in the expansion using basis $\{\omega_k\}_{k=1}^n$ and $x$ is supported on $N$ in the expansion using basis $\{\tau_j\}_{j=1}^n$, then $x=0$. 	
\end{theorem}
Observe that the constant 
\begin{align*}
 C_d e^{C_d \min \{m(E)m(F), m(E)^\frac{1}{d}w(F), m(F)^\frac{1}{d}w(E)\}} 	
\end{align*}
in Inequality (\ref{NAZAROVJAMINGIN})  is depending upon subsets $E$, $F$ and not on the entire domain $\mathbb{R}$ of functions $f$, $\widehat{f}$. Thus it is natural to ask whether there is a constant sharper in Inequality (\ref{FGJI}) depending upon subsets $M$, $N$ and not on $\{1, \dots, n\}$. A careful observation in the  proof of Theorem  \ref{FGJ} gives following result. 
\begin{theorem}
Let $(\{f_j\}_{j=1}^n, \{\tau_j\}_{j=1}^n)$	and $(\{g_k\}_{k=1}^n, \{\omega_k\}_{k=1}^n)$ be p-orthonormal bases  for $\mathcal{X}$.  If $M,N\subseteq \{1, \dots, n\}$ are such that 
\begin{align*}
	o(M)^\frac{1}{q}o(N)^\frac{1}{p}< \frac{1}{\displaystyle \max_{j\in M, k\in N}|g_k(\tau_j) |},
\end{align*}
then for all $x \in \mathcal{X}$, 	
\begin{align*}
	\|x\|\leq \left(1+\frac{1}{1-o(M)^\frac{1}{q}o(N)^\frac{1}{p}\displaystyle\max_{j \in M, k\in N}|g_k(\tau_j)|}\right)\left[\left(\sum_{j\in M^c}|f_j(x)|^p\right)^\frac{1}{p}+\left(\sum_{k\in N^c}|g_k(x) |^p\right)^\frac{1}{p}\right].
\end{align*}	
\end{theorem}
Similarly we have the following result from Theorem \ref{FGJ2}.
\begin{theorem}
Let $(\{f_j\}_{j=1}^n, \{\tau_j\}_{j=1}^n)$	and $(\{g_k\}_{k=1}^n, \{\omega_k\}_{k=1}^n)$ be p-orthonormal bases  for $\mathcal{X}$.  If $M,N\subseteq \{1, \dots, n\}$ are such that 
\begin{align*}
	o(M)^\frac{1}{q}o(N)^\frac{1}{p}< \frac{1}{\displaystyle \max_{j\in N, k \in M}|f_j(\omega_k) |},
\end{align*}
then for all $x \in \mathcal{X}$, 	
\begin{align*}
	\|x\|\leq \left(1+\frac{1}{1-o(M)^\frac{1}{q}o(N)^\frac{1}{p}\displaystyle\max_{j\in N, k \in M}|f_j(\omega_k)|}\right)\left[\left(\sum_{k\in M^c}|g_k(x)|^p\right)^\frac{1}{p}+\left(\sum_{j\in N^c}|f_j(x) |^p\right)^\frac{1}{p}\right].
\end{align*}	
\end{theorem}
Theorem  \ref{FGJ}  brings the following question.
\begin{question}
	Given $p$ and a Banach space $\mathcal{X}$ of dimension $n$, for which subsets $M,N\subseteq \{1, \dots, n\}$ and pairs of p-orthonormal bases $(\{f_j\}_{j=1}^n, \{\tau_j\}_{j=1}^n)$,  $(\{g_k\}_{k=1}^n, \{\omega_k\}_{k=1}^n)$ for $\mathcal{X}$, we have equality in Inequality (\ref{FGJI})?
\end{question}
It is clear that we used $1<p<\infty$ in the proof of  Theorem  \ref{FGJ}. However, Definition \ref{PONB} can easily be extended to include cases  $p=1$ and $p=\infty$. This therefore leads to the following question. 
\begin{question}
	Whether there are Functional Ghobber-Jaming  Uncertainty Principle (versions of Theorem  \ref{FGJ}) for 1-orthonormal bases and $\infty$-orthonormal bases?
\end{question}
We end by mentioning that Donoho-Stark-Elad-Bruckstein-Ricaud-Torrésani Uncertainty Principle	for finite dimensional Banach spaces is derived in \cite{KRISHNA3} (actually, in  \cite{KRISHNA3} the functional uncertainty principle was derived for p-Schauder frames which is general than p-orthonormal bases. Thus it is worth to derive Theorem   \ref{FGJ} or a variation of it  for p-Schauder frames, which we are unable).

 \bibliographystyle{plain}
 \bibliography{reference.bib}

\end{document}